\documentclass[11pt]{amsart}

\usepackage{amsmath,amssymb}
\usepackage{amsthm}
\usepackage{enumitem}
\usepackage{multicol}
\usepackage{enumitem}

\usepackage{color}

\usepackage{tikz}
\usetikzlibrary{matrix,arrows}

\newtheorem{theorem}{Theorem}

\newtheorem{proposition}[theorem]{Proposition}

\theoremstyle{definition}
\newtheorem{definition}[theorem]{Definition}

\newtheorem{claim}[theorem]{Claim}

\newcommand{\Q}{\mathbb{Q}}

\newcommand{\Z}{\mathbb{Z}}
\newcommand{\F}{\mathbb{F}}

\newcommand{\la}{\langle}
\newcommand{\ra}{\rangle}

\newcommand{\CC}{\mathcal{C}}
\newcommand{\B}{\mathcal{B}}

\title{On the Regularity of Orientable Matroids}

\begin{document}
\maketitle
\begin{center}
\author{Libby Taylor}\footnotemark[1]
\end{center}

\footnotetext[1]{
School of Mathematics, Georgia Institute of Technology
(libbytaylor@gatech.edu)
}

\begin{abstract}
We present two characterizations of regular matroids among orientable matroids and use them to give a measure of ``how far" an orientable matroid is from being regular.
\end{abstract}

\section{Introduction}\label{intro}

A regular matroid is a matroid which is representable over any field; these objects have been heavily studied in the literature.  
The following are both characteristics of a regular matroid $M$:

\begin{enumerate}[noitemsep]
\item Every basis of $M$ generates the same lattice of fundamental circuits.
\item The rank of the circuit lattice of $M$ is equal to $\text{corank}(M)$.
\end{enumerate}

These are both folklore results which, to the best of the author's knowledge, have not been written down explicitly in the literature.  A proof of these statements is provided in Section \ref{characterizations}.  

Our main contribution in this paper is to prove that the two properties above characterize regular matroids.  (2) was proved in ~\cite{nickel}; here we give a new proof.

\begin{theorem}\label{circuitlattices}

For $M$ an orientable matroid, the following are equivalent:

\begin{enumerate}[noitemsep]
\item $M$ is regular.
\item Every basis of $M$ generates the entire circuit lattice of $M$.
\item The rank of the circuit lattice of $M$ is equal to $\text{corank}(M)$.  
\end{enumerate}

\end{theorem}

Although there are several other characterizations of regularity among orientable matroids--see for example \cite{chiho}, \cite{goldenratio}, and \cite{tutte}--the ones given here can be used to provide a notion of how badly the matroid fails to be regular.  This will be made precise in Section \ref{distance}.  

\section{Background}\label{background}

In this section we recall the definitions of oriented and regular matroids and set up notation.  For a more comprehensive overview, we refer the reader to the book on matroids by Oxley \cite{oxley} and the book on oriented matroids by Bj\"orner et al \cite{redbook}.  An oriented matroid $M$ is, informally speaking, a matroid together with a set of sign data which behaves well under circuit elimination.  More precisely, when $E$ is the ground set of $M$, a signed subset of $E$ is a map $X:E\to \{+,-,0\}$.  The {\em support} of $X$ is denoted $\underline{X}$ and defined to be $\{e\in E:X(e)\neq 0\}$.  We define $X^+=\{e\in E:X(e)=+\}$ and $X^-=\{e\in E:X(e)=-\}$.  

\begin{definition}\label{orientedmatroid}

An oriented matroid $M=(E,\CC)$ is a nonempty finite set $E$ together with a collection $\CC$ of signed subsets that satisfy the following axioms:
\begin{enumerate}[noitemsep]
\item $\CC\neq \emptyset$.
\item If $C\in \CC$, then $-C\in \CC$.
\item For all $C_1,C_2\in \CC$, if $\underline{C_1}\subseteq \underline{C_2}$, then either $C_1=C_2$ or $-C_1=C_2$.
\item For all $C_1,C_2\in \CC$, $e\in C_1^+\cap C_2^-$, and $f\in (C_1^+\setminus C_2^-)\cup (C_1^-\setminus C_2^+)$, there is a $C_3\in \CC$ such that $C_3^+\subset (C_1^+\cup C_2^+)\setminus \{e\}$, $C_3^-\subset(C_1^-\cup C_2^-)\setminus \{e\}$, and $f\in \underline{C_3}$. \footnote{This condition is generally known as the \emph{strong circuit elimination} axiom.}
\end{enumerate}

\end{definition}

The elements of $\CC$ are called {\em signed circuits}.  

A matrix is said to be unimodular if the determinant of every maximal square minor is $\pm 1$; a matrix is totally unimodular if the determinant of every square minor is $\pm 1$.  

\begin{theorem}\cite[Theorem 3.1.1]{white}\label{unimodular}

Let $M$ be a matroid.  The following are equivalent:

\begin{enumerate}[noitemsep]
\item $M$ is representable over $\Q$ by a totally unimodular matrix.
\item $M$ is representable over $\Q$ by a unimodular matrix.
\item $M$ is representable over any field $\F$.
\item $M$ is orientable and representable over $\F_2$.
\end{enumerate}

\end{theorem}

If $M$ satisfies any of the above conditions, we say $M$ is regular.

Since matroids representable over ordered fields are orientable, it follows that all regular matroids are orientable. 
In fact, the oriented structure on a regular matroid is unique up to reorientation \cite[Corollary 7.9.4]{redbook}.  It is well known that a matroid is regular if and only if its dual is.

Given a basis $B$ of $M$ and $e\in E\setminus B$, denote as $C(B,e)$ the oriented fundamental circuit of $B$ and $e$; that is, the unique oriented circuit whose support is contained in $B\cup e$.

Throughout the rest of the paper, $M$ will denote an orientable matroid of rank $r$ with ground set $E$.  Let $\B$ denote the set of bases of $M$ and let $C(B_i)$ and $C^*(B_i)$ denote the set of oriented fundamental circuits and cocircuits, respectively, of a basis $B_i\in \B$.  Let $\Lambda_i$ and $\Lambda_i^*$ denote the lattices in $\Z^E$ generated by $C(B_i)$ and $C^*(B_i)$, respectively.  The lattice $\Lambda_i$ is the \textit{circuit lattice} of $M$, and $\Lambda_i^*$ is the \textit{cocircuit lattice} of $M$.  In defining these lattices, vectors with entries in $\{+,-,0\}$ are treated as vectors in $\Z^E$ with entries in $\{1,-1,0\}$ in the obvious way.

\section{Characterizations of regular matroids}\label{characterizations}

We begin by proving the folklore result mentioned in Section \ref{intro}.

\begin{proposition}\label{folklore}

When $M$ is a regular oriented matroid, the following hold.

\begin{enumerate}[noitemsep]
\item Every basis of $M$ generates the same lattice of fundamental circuits.
\item The rank of the circuit lattice of $M$ is equal to $\text{corank}(M)$.  That is, $\text{rank}(\Lambda)=\text{corank}(M)$.  
\end{enumerate}

\end{proposition}

\begin{proof}

Fix a totally unimodular matrix representation $A$ for $M$.  We first prove property (1).  The set of circuits of $M$ is defined to be the support-minimal elements of $\text{ker}(A)\cap \Z^E$.  Clearly $\text{ker}(A)$ has dimension equal to $\text{corank}(M)$ over $\Q$.  Any set of fundamental circuits of a single basis are independent over $\Q$, so they generate all of $\text{ker}(A)$.  

\medskip

Totally unimodular matrices have totally unimodular kernels (see for example ~\cite{tumatrix}, Lemma 12), by which we mean that there exists a matrix whose rows are a basis for the kernel which is itself totally unimodular.  Fix a basis $B$ and let $C(B)$ its set of fundamental circuits.  Consider the matrix whose rows are the elements of $C(B)$; since $C(B)$ forms a basis for $\text{ker}(A)$, this matrix is equivalent under a change of basis to one which is totally unimodular.  Since the fundamental circuits in $C(B)$ form a unimodular matrix and they form a $\Q$-basis for $\text{ker}(A) \cap \Z^E$, they must also form an integral basis for $\text{ker}(A)\cap \Z^E$.  Since $C(B)$ generates $\text{ker}(A)\cap \Z^E$, and $B$ was arbitrary.  This completes the proof of (1).  

\medskip

To prove (2), observe that (1) implies that the maximum size of an integrally independent set of circuits is $\text{corank}(M)$, since this is the rank of the lattice of circuits of $M$.

\end{proof}

Now we turn to proving that these properties of regular matroids are in fact characterizations.

\begin{proof}[Proof of Theorem~\ref{circuitlattices}]

We begin by proving the first property in Theorem \ref{circuitlattices}.  Proposition \ref{folklore} gives one direction, so it remains here to show that if each basis generates the entire circuit lattice, then $M$ is regular.  Suppose that all the $\Lambda_i$ are equal, and fix an initial basis $B_0$.  (Recall that $\Lambda_i$ denotes the lattice generated by the fundamental circuits of a basis $B_i$.)  Let $A$ be the matrix defined by
\[
A=[I_r|D]
\]
where $r$ is the rank of $M$, the columns in $I_r$ (which denotes the $r\times r$ identity matrix) correspond to the elements of $B_0$, and the columns in $D$ correspond to $e\in E\setminus B_0$.  The columns of $D$ are constructed according to their oriented fundamental circuit in $B_0$.  That is, the linear dependence among the columns of $I_r$ and a column $e\in D$ produces the oriented fundamental circuit $C(B_0,e)$, and the column $e$ has entries in $\{0, \pm 1\}$.

The matrix $A$ defines some matroid, which we call $M'$.  The circuits of $M'$ are the signed support of support-minimal elements of $\text{ker}(A)\cap \Z^E$.  The proof will proceed in 3 steps.  First, we will show that $C(M) \subseteq C(M')$.  Second, we will show that $A$ is unimodular, and therefore $M'$ is regular.  Third, we show that $C(M) \supseteq C(M')$, which implies that $M=M'$ and therefore that $M$ is regular.  For the first step, consider the following

\begin{claim}
$C(M) \subseteq C(M')$.  
\end{claim}

\begin{proof} By assumption, $C(B_0)$ generates the circuit lattice of $M$; and by construction, $C(B_0)\subseteq C(M')$.  Therefore, any circuit of $M$ must be a linear dependence in $A$, since it is in the lattice of circuits of $M'$.  It remains to show that this linear dependence in $A$ is support-minimal in $M'$, i.e. support-minimal in $\text{ker}(A) \cap \Z^E$.  Suppose this circuit is not support-minimal, so there exists a circuit $C$ of $M$ and a circuit $C'$ of $M'$ such that $\text{supp}(C')\subsetneq \text{supp}(C)$.  Then $C'$ is a support-minimal element of $\text{ker}(A) \cap \Z^E$.  Since it is an element of $\text{ker}(A)$, it is some $\Q$-linear combination of elements of $C(B_0)$, and can be written as $C'=q_1 c_1 +\cdots+q_lc_\ell$ where the $q_i\in \Q$ and the  $c_i$ are fundamental circuits of $B_0$.  Therefore there exists an integer $t$ such that $tC'$ is in the circuit lattice of $M$ (we can take $t$ to be the least common denominator of the $q_i$ so that $tC'$ is an integer linear combination of the $c_i$), which witnesses that it is a dependent set in $M$ whose support is properly contained in that of $C$.  This contradicts the support-minimality of $C$ in $M$.  

Therefore $C(M)\subseteq C(M')$.  

\end{proof}

We now turn to the second step of the proof:

\begin{claim}\label{unimodular}

$A$ is unimodular, and therefore $M'$ is regular.

\end{claim}

\begin{proof}

First, we will prove that it suffices to show that every support-minimal linear dependence of columns of $A$ can be written with coefficients of absolute value 1.  If this is the case, then we can apply Cramer's rule to prove the unimodularity of $A$.  

Let $N$ be a square submatrix of $A$ which has rank $r$.  We will say that the columns of such an $N$ form a basis for $A$.  If in addition $\text{det}(N)=\pm 1$, we will say that $N$ is a unimodular basis for $A$.  
Apply Cramer's rule to a submatrix $N$ of $A$ whose columns form a basis for $A$ and a column $b$ of $A$ which is not in $N$.  Then there exists a unique linear dependence among the columns of $N$ and $b$, which is the fundamental circuit $C(N,b)$.  Assume for the moment that the coefficients of this dependence are in $\{1,0,-1\}$.  If $\det(N)=\pm 1$, then $\det(N_b^i)\in \{1,0,-1\}$ for each $i$, where $N_b^i$ denotes the matrix obtained from $N$ by replacing the $i^{th}$ column of $N$ with $b$.  This proves that any basis which can be obtained from a unimodular basis via a single basis exchange move is itself unimodular.  Any basis can be obtained from any other via a sequence of basis exchange moves, so it is sufficient that there exist some unimodular basis of $A$.  Since $A$ contains a full-rank identity matrix as a submatrix, this holds.  Therefore it remains to prove that every support-minimal linear dependence of columns of $A$ can be written with coefficients of absolute value 1.

Suppose there is some support-minimal linear dependence in $A$ which cannot be expressed in this way, so we have $a_1e_1+\dots+a_ke_k=0$ for the $a_i$ nonzero integers not all having the same absolute value and the $e_i$ columns of $A$.  (By abuse of notation, we let $e_1,\dots,e_k$ denote any set of $k$ distinct columns of $A$; these need not be the first $k$ columns of $A$.)  This linear dependence is in the kernel of $A$, so it can be written as some $\Z$-linear combination of the fundamental circuits associated to $B_0$, since these form a basis for $\text{ker}(A)\cap \Z^E$.  Therefore $a_1e_1+\dots+a_ke_k$ is a linear combination of circuits in $C(B_0)$, so the signed support of this linear combination forms an oriented vector (that is, a linear combination of circuits) of $M$, which we will denote $C_k$.  Therefore the support of $C_k$ is a dependent set in $M$, and we divide into two cases.  First, we assume that $C_k$ is actually a circuit of $M$, in which case the linear dependence in $M$ is actually support-minimal.

$C_k$ is a circuit of $M'$ by definition.  By hypothesis, all $\Lambda_i$ are equal, so $C_k$ is in the $\Z$-span of the fundamental circuits of $B_0$, that is, in $\Lambda_0$.  Then we have that $C_k-(a_1e_1+\dots+a_ke_k)$ is also in the $\Z$-span of elements of $C(B_0)$.  By choosing an appropriate integer multiple $t$ of $C_k$, at least one term $e_i$ in $tC_k-(a_1e_1+\dots+a_ke_k)$ can be made to cancel.  But this gives some new circuit of $M'$ whose support is properly contained in the support of $C_k$, contradicting the support-minimality of $C_k$.

Next, suppose $C_k$ is an oriented vector in $M$ but not a circuit.  Then there is some circuit of $M$ whose support is properly contained in that of $C_k$; call this circuit $C_j$.  Since $C(M)\subseteq C(M')$, we would have $C_j$ also a circuit of $M'$, contradicting the support-minimality of $C_k$.

\medskip

This gives that every support-minimal linear dependence of columns of $A$ can be written such that the coefficients in its support have absolute value 1.  Combined with Cramer's rule, this proves the unimodularity of $A$.  Since $A$ is unimodular, by Theorem \ref{unimodular} $M'$ must be regular.  

\end{proof}

It remains to prove the third step, that $M=M'$; equivalently, we can show that the circuits of $M$ are exactly the circuits of $M'$.  

\begin{claim}
$C(M) \subseteq C(M')$.
\end{claim}

\begin{proof} All $\Lambda_i$ are equal by hypothesis, and all $\Lambda_j'$ are equal by the fact that $M'$ is regular.  (Here, we denote by $\Lambda_j'$ the lattice generated by the fundamental circuits of a basis $B_j'$ of $M'$.)  Moreover, $\Lambda_0=\Lambda_0'$ by construction of $A$.  Therefore, $M$ and $M'$ have the same circuit lattice and a common basis, so $C(M)=C(M')$.  Since $C(M')$ is the set of support-minimal elements of $\text{ker}(A) \cap \Z^E$, so is $C(M)$.  This implies that $A$ is a totally unimodular representation for $M$, so $M$ is regular. 

\end{proof}

This completes the proof of the equivalence of (1) and (2) in Theorem \ref{circuitlattices}.


Note that the dual statement for lattices generated by fundamental cocircuits of bases is also true.  That is, $M$ is regular if and only if all the $\Lambda_i^*$ are the same.

We now turn to proving the equivalence of (1) and (3) of Theorem \ref{circuitlattices}:  an orientable matroid $M$ is regular if and only if $\text{rank}(\Lambda)=\text{corank}(M)$.


Proposition \ref{folklore} gives that if $M$ is regular, then $\text{rank}(\Lambda)=\text{corank}(M)$.  
So, suppose $M$ is nonregular.  By (2) of Theorem \ref{circuitlattices}, there are two bases, call them $B_1$ and $B_2$, which do not generate the same circuit lattice.  Then there is an oriented fundamental circuit $c$ of $B_2$ which is not in the integral span of the circuits of $B_1$.  In order to produce an independent set of circuits of size $\text{corank}(M)+1$, we will show that there is no circuit in $c_1\in C(B_1)$ which is in the integral span of $(c\cup C(B_1))\setminus c_1$.  This will prove that the minimum size of a generating set for the lattice generated by $C(B_1)\cup c$ is at least $\text{corank}(M)+1$.  Suppose to the contrary that
\[
c_1=zc+z_2c_2+\dots+z_kc_k
\]
for $z,z_2,\dots,z_k\in \Z$, $c_1\in C(B_1)$, and $c_2,\dots,c_k\in C(B_1)\setminus \{c_1\}$.  Since by assumption $c$ is not in the $\Z$-span of $C(B_1)$, we must have $|z|>1$.  

Since $c_1$ is a fundamental circuit of $B_1$, there exists a unique $e_1\in E\setminus B_1$ which is in the support of $c_1$.  This $e_1$ is not in the support of any of $c_2,\dots,c_k$, since these are fundamental circuits of $B_1$ distinct from $c_1$.  Therefore, in order for the integral dependence to hold, $e_1$ must be in the support of $c$.  Since all nonzero entries in $c_1$ have absolute value 1, this implies that $|z|=1$, contradicting the initial assumption that $c$ is not in the $\Z$-span of $C(B_1)$.  Therefore, $C(B_1)\cup c$ forms an integrally independent set of circuits of size $\text{corank}(M)+1$, proving that $M$ is regular if and only if $\text{rk}(\Lambda)=\text{cork}(M)$.



\end{proof}

Note that the maximum size of a set of integrally independent circuits of $M$ is equal to the rank of the circuit lattice of $M$.  Since the set of fundamental circuits of a single basis is always linearly independent, it will always be the case that $\text{rank}(\Lambda)\ge \text{corank}(M)$, and (3) of Theorem \ref{circuitlattices} is equivalent to the statement that this inequality is actually an equality if and only if $M$ is regular.

As a corollary of Theorem \ref{circuitlattices}, we observe that the Jacobian of a matroid is well-defined if and only if the matroid is regular.  There are several equivalent definitions of the Jacobian: Biggs proves in \cite{biggs} that when $G=(V,E)$ is a graph, the following groups are all isomorphic and have order equal to the number of spanning trees of the graph:
$$\frac{\Z^E}{\Lambda\oplus \Lambda^*} \cong \frac{\Lambda^\#}{\Lambda} \cong \frac{\Lambda^{*\#}}{\Lambda^*}$$
where $\Lambda^\#$ denotes the dual lattice to $\Lambda$; that is, $\Lambda^\#=\{y\in \Q^E:\la x, y \ra \in \Z^E \text{ for all }x\in \Lambda\}$.  Each of these isomorphic groups is often referred to as the \textit{Jacobian} of the graph.

The proofs of these facts go through mutatis mutandis for regular matroids.  This means that the Jacobian of a regular matroid can be defined in the same way as for a graph, and its cardinality will be equal to the number of bases of the matroid.  For nonregular matroids, though, none of these groups will have cardinality equal to the number of bases.  The fact that these groups have the ``right'' cardinality in the regular case relies crucially on two facts: first, that the fundamental circuits of any basis are an integral basis for the circuit space of the matroid, and that each circuit is orthogonal to each cocircuit.  We have shown that these facts never hold in the nonregular case, which explains why the cardinality of these groups will never be ``right''; that is, the cardinality will not be equal to the number of bases.  This explains why none of the natural candidates for the Jacobian of a nonregular matroid have the properties we would desire from an analogy with the regular case.

\section{Distance from regularity}\label{distance}

The characterizations of regularity given in Section \ref{characterizations} are certainly not an exhaustive list; others can be found in \cite{chiho}, \cite{goldenratio}, \cite{tutte} and \cite{seymour}.  Most such characterizations come in two flavors.  The first is a condition that a matroid is regular if and only if it is representable over some finite set of fields.  For example, an orientable matroid is regular if and only if is is representable over $\F_2$ (see (4) of Theorem \ref{unimodular}).  
The other flavor is that of excluded minor characterizations.  In \cite{tutte}, Tutte proves that a matroid is regular if and only if it does not contain a minor isomorphic to either the Fano matroid or its dual.  It is also true that an orientable matroid is regular if and only if it contains no minor isomorphic to the uniform matroid $U_4^2$.  (This follows from the facts that an oriented matroid is regular if and only if it is binary, and a matroid is binary if and only if it contains no $U_4^2$ minor.)  Yuen uses this last characterization in \cite{chiho} to prove that an oriented matroid $M$ is regular if and only if the number of circuit-cocircuit equivalence classes of orientations is equal to the number of bases.  
The characterizations in Theorem \ref{circuitlattices} can be used to provide a measure of ``how badly" $M$ fails to be regular.  We use (3) of this theorem to provide an irregularity parameter for an oriented matroid, and prove that it is monotone under taking minors.

We define this irregularity parameter to be the quantity $\rho(M)=\text{rank}(\Lambda)- \text{corank}(M)$.  We have that $\rho(M)=0$ if and only if $M$ is regular, so the value $\rho(M)$ may be thought of as a measure of how far $M$ is from being regular.  

\begin{proposition}

$\rho(M)$ is additive on direct sums and monotone under taking minors.

\end{proposition}

\begin{proof}

Suppose $M\cong M_1\oplus M_2$.  Both the bases and the circuits of $M$ can be separated into the summands $M_1$ and $M_2$.  The maximum number of integrally independent circuits of $M$ is equal to $\text{rank}(\Lambda(M_1))+\text{rank}(\Lambda(M_2))$, and $\text{corank}(M)=\text{corank}(M_1)+\text{corank}(M_2)$.  This immediately implies the first statement.

\medskip

For the second statement, consider first the behavior of $\rho(M)$ on the contraction of some $e\in E$, denoted $M/e$.  Assume first that $e$ is not a loop.  Contracting $M$ decreases the rank of the matroid by 1 but leaves the corank unchanged; that is, $\text{rank}(M)=\text{rank}(M/e)+1$ and $\text{corank}(M)=\text{corank}(M/e)$.  It is obvious that the number of integrally independent circuits cannot increase under contraction, since any maximal independent set of circuits of $M$ will descend to a maximal independent set of circuits of $M/e$.  Therefore $\rho(M/e)\le \rho(M)$.  In the case that $e$ is a loop, $\text{rank}(M/e)=\text{rank}(M)$ and $\text{corank}(M/e)=\text{corank}(M)$, so $\rho(M/e)=\rho(M)$.  

\smallskip

Deleting an element $e\in M$ is equivalent to contracting the corresponding element of $M^*$.  Therefore, the monotonicity of $\rho(M)$ on deletion follows from the monotonicity of $\rho(M')$ on contraction.


\end{proof}

Since $\rho$ is monotone under taking minors, the class of matroids such that $\rho(M)\le n$ for any nonnegative integer $n$ is a minor-closed family.  Clearly the only forbidden minors for $\rho(M)=0$ are all orientations of $U_4^2$, since any oriented matroid with $\rho(M)>0$ is nonregular and therefore contains a $U_4^2$ minor.  An easy computation shows that $\rho(U_4^2)=2$, since the four circuits of $U_4^2$ form an integrally independent set and corank$(U_4^2)=2$.  Therefore every nonregular matroid $M$ has $\rho(M)\ge 2$.  It could be interesting to investigate whether there is a finite list of forbidden minors characterizing the property $\rho \le n$ for an integer $n$.  

\section{Acknowledgements}

The author would like to thank Matt Baker for many helpful discussions and for helpful feedback on an earlier version of this paper.


\begin{thebibliography}{99}
    \bibitem{oxley}
    {Oxley, Matroid Theory, volume 3., Oxford University Press,     USA, 2006}
    \bibitem{nickel}
    {R. Nickel.  Flows and Colorings in Oriented Matroids, 2012 (Ph.D. thesis).}
    \bibitem{redbook}
    {Anders Bj\"orner, Michael Las Vergnas, Bernd Sturmfels, Neil White, and G\"unter Ziegler, Oriented Matroids, volume 46 of Encyclopedia of Mathematics and its Applications, Cambridge University Press, Cambridge, 1999}
    \bibitem{chiho}
    {C.H. Yuen, On the number of circuit-cocircuit reversal classes of an oriented matroid, Preprint.  Available at https://arxiv.org/pdf/1707.00342.pdf}
    \bibitem{goldenratio}
    {R.A. Pendavingh and S.M.H. van Zwam, Lifts of matroids over partial fields.  J. Comb. Theory Series B, vol. 100, issue 1, pp. 26-67.  2010.}
    \bibitem{tutte}
    {W.T. Tutte, Lectures on Matroids.  J. Res. Nat. Bur. Standards, Sect. B, vol. 69B, pages 1-47.  1965.}
    \bibitem{white}
    {N. White.  Combinatorial geometries, volume 29 of of Encyclopedia of Mathematics and its Applications.  Cambridge University Press, 1987.}
    \bibitem{tumatrix}
    {J. Bader, R. Hildebrand, R. Weismantel, R. Zenklusen, Mixed integer reformulations of integer programs the the affine TU-dimension of a matrix.  Mathematical Programming, vol. 169, issue 2, pp 565-584.  2017.}
    \bibitem{seymour}
    {P.D. Seymour, Decomposition of regular matroids.  J. Comb. Theory Series B, vol. 28, issue 3, pp. 305-359.  1980.}
    \bibitem{biggs}
    {N. Biggs, Algebraic potential theory on graphs.  Bulletin of the London Mathematical Society, volume 29 issue 6, 1997.}
\end{thebibliography}
\end{document}